\documentclass[reqno,11pt,letterpaper]{amsart}

\usepackage{amsmath,amssymb,amsthm,graphicx,mathrsfs,url}
\usepackage[usenames,dvipsnames]{color}
\usepackage
{hyperref}
\usepackage{amsxtra}
\usepackage{graphicx}
\usepackage{tikz}
\relax

\numberwithin{equation}{section}

\newcommand{\be}{\begin{equation}}
\newcommand{\ee}{\end{equation}}

\newcommand{\loc}{\operatorname{loc}}

\theoremstyle{plain}

\newtheorem{thm}{Theorem}
\newtheorem{prop}{Proposition}[section]
\newtheorem{cor}[prop]{Corollary}
\newtheorem{lem}[prop]{Lemma}

\theoremstyle{definition}

\newtheorem{rem}[prop]{Remark}

\numberwithin{equation}{section}

\def\squarebox#1{\hbox to #1{\hfill\vbox to #1{\vfill}}}

\usepackage{amsxtra}

\ifx\pdfoutput\undefined
\DeclareGraphicsExtensions{.pstex, .eps}
\else
\ifx\pdfoutput\relax
\DeclareGraphicsExtensions{.pstex, .eps}
\else
\ifnum\pdfoutput>0
\DeclareGraphicsExtensions{.pdf}
\else
\DeclareGraphicsExtensions{.pstex, .eps}
\fi
\fi
\fi

\title[Concerning the pathological set in probabilistic well-posedness 
]
{
Concerning the pathological set in the context of probabilistic well-posedness 
}
\author[C-M. Sun]{Chenmin Sun}
\address{Universit\'e de Cergy-Pontoise, Laboratoire de Math\'ematiques AGM, UMR  8088 du CNRS, 2 av. Adolphe Chauvin
	95302 Cergy-Pontoise Cedex, France }
\email{chenmin.sun@u-cergy.fr}

\author[N. Tzvetkov]{Nikolay Tzvetkov}
\address{Universit\'e de Cergy-Pontoise, Laboratoire de Math\'ematiques AGM, UMR  8088 du CNRS, 2 av. Adolphe Chauvin
95302 Cergy-Pontoise Cedex, France }
\email{nikolay.tzvetkov@u-cergy.fr}

\usepackage{amssymb}
\usepackage{amsmath, amsthm, amsopn, amsfonts}

\def\11{{\rm 1~\hspace{-1.4ex}l} }
\def\R{\mathbb R}

\def\N{\mathbb N}

\def\T{\mathbb T}

\begin{document}

	\begin{abstract}
	 We prove a complementary result to the probabilistic well-posedness for the nonlinear wave equation.
	 More precisely, we show that there is a dense set $S$ of the Sobolev space of super-critical regularity such that (in sharp contrast with the probabilistic well-posedness results) the family of global smooth solutions, generated by the convolution with some approximate identity of the elements of $S$, does not converge in the space of super-critical Sobolev regularity.
	 \\ 
	 %The key ingredient is an idea of Lebeau (\cite{Le01}) using finite propagation speed.
	%\end{abstract}   
	
	%\begin{abstract}
	{
\large R\'esum\'e.} On d\'emontre un r\'esultat  compl\'ementaire \`a ceux manifestant le caract\`ere bien pos\'e probabiliste de l'\'equation des ondes avec des donn\'ees initiales de r\'egularit\'e de Sobolev super critique par rapport au changement d'\'echelle laissant invariant l'\'equation.  
	\end{abstract}

	\maketitle 
	\setlength{\parskip}{0.3em}  
%	\tableofcontents
	
	\section{Introduction}   
	\label{in}
In this work, we are interested in the three dimensional nonlinear wave equation
\begin{equation}\label{wave:main}
\begin{cases}
& \partial_t^2u-\Delta u+|u|^{2\sigma}u=0,\quad (t,x)\in\R\times\T^3,\\
&(u,\partial_tu)|_{t=0}=(f,g)\in\mathcal{H}^s(\T^3),
\end{cases}
\end{equation}
where $u$ is a real-valued function and 
$$ 
\mathcal{H}^s(\T^3):=H^s(\T^3)\times H^{s-1}(\T^3).
$$
The nonlinear wave equation \eqref{wave:main} is a Hamiltonian system with conserved energy
$$ H[u]:=\frac{1}{2}\int_{\T^3}|\nabla u|^2dx+\frac{1}{2\sigma+2}\int_{\T^3}|u|^{2\sigma+2}dx.
$$
It was shown (see \cite{Gri,ShSt}) that when $\sigma\leq 2$, the problem \eqref{wave:main} possesses a global strong solution in the energy space $\mathcal{H}^1(\T^3)$. By replacing $\T^3$ to $\R^3$, the scaling
$$ u\mapsto u_{\lambda}(t,x):=\lambda^{\frac{1}{\sigma}}u(\lambda t,\lambda x)
$$
keeps the equation \eqref{wave:main} invariant. This leads to the critical regularity index $s_c=\frac{3}{2}-\frac{1}{\sigma}\leq 1$. Intuitively, for $s<s_c$ if the initial data is concentrated at the frequency scale $\gg 1$ and is of size $1$ measured by the $\mathcal{H}^s$ norm, then the nonlinear part in the dynamics of \eqref{wave:main} is dominant and it causes instability of the $\mathcal{H}^s$ norm of the solution. This is called a norm inflation and it was extensively studied, see  \cite{CCTao},\cite{Le01},\cite{Le05} in the context of nonlinear wave equations.  For instance, it was shown in \cite{CCTao} that there exists a sequence of smooth initial data whose ${\mathcal H}^s$ norms converge to zero, while the ${\mathcal H}^s$ norms of the obtained sequence of solutions amplifies  at very short time. We also refer to \cite{Li} where a different concentration phenomenon, related to the Lorentz invariance of the wave equation, is observed. 

In \cite{BTz08} and \cite{BTz14}, by using probabilistic tools,  N.~Burq and the second author showed that  problem \eqref{wave:main} with cubic nonlinearity still possesses global strong solutions for a "large class" of functions of super-critical regularity.  
The result was  further extended to $1\leq\sigma\leq 2$ in \cite{OhPo} and \cite{SXia}.  More precisely, the following statement follows from  \cite{BTz14},\cite{OhPo},\cite{SXia}.
 \begin{thm} 
 %[\cite{BTz14},\cite{OhPo},\cite{SXia}]
 \label{thmproba}
 Let $1\leq\sigma\leq 2$ and $1-\frac{1}{\sigma}<s<s_c=\frac{3}{2}-\frac{1}{\sigma}$. Then there is a dense set $\Sigma\subset\mathcal{H}^s(\T^3)$ satisfying 
 $ \Sigma \cap \mathcal{H}^{s'}(\T^3)=\emptyset$   for every  $s'>s$  such that the following holds true. 
 For every $(f,g)\in\Sigma$,  let $(f_{n},g_{n})$ be the sequence in $ C^{\infty}(\T^3)\times C^{\infty}(\T^3)$ defined by the regularization by convolution, i.e.
 $$ 
 f_{n}=\rho_{n}\ast f,\quad g_{n}=\rho_n\ast g,
 $$
 where $(\rho_n)_{n\in\N}$ is an approximate identity. 
 Denote by $(u_n(t),\partial_tu_n(t))$ the smooth solutions of \eqref{wave:main} with the smooth initial data $(f_{n},g_{n})$. Then there exists a limit object $u(t)$ such that for any $T>0$,
 $$ \lim_{n\rightarrow\infty}\big\|(u_n(t),\partial_tu_n(t))-(u(t),\partial_tu(t)) \big\|_{L^{\infty}([-T,T];\mathcal{H}^s(\T^3) )}=0.
 $$
 Moreover $u(t)$ solves \eqref{wave:main} in the distributional sense.
 \end{thm}  
When $1\leq\sigma<2$, the above theorem can be extended to $s=1-\frac{1}{\sigma}$, thanks to \cite{BTz14} (the case $\sigma=1$) and a recent result \cite{La18}(the case $1<\sigma<2$). 

In Theorem~\ref{thmproba} the set  $\Sigma$ is a full measure set with respect to a suitable non degenerate probability measure $\mu$ on the Sobolev space $\mathcal{H}^s(\T^3)$ such that $\mu(\mathcal{H}^{s'}(\T^3))=0$  for every  $s'>s$ .  One proves more than Theorem~\ref{thmproba} in  \cite{BTz14},\cite{OhPo},\cite{SXia} but the statement of Theorem~\ref{thmproba} is the suitable one for our purpose here. 

Theorem~\ref{thmproba} is inspired by the seminal contribution of Bourgain \cite{Bo}. There are however several new features with respect to \cite{Bo}. The first one is that more general randomisations compared to \cite{Bo} are allowed.  This led to results similar to Theorem~\ref{thmproba} in the context of a non compact spatial domains (see e.g. \cite{BOP}, \cite{LM}).  Next, the argument allowing to pass from local to global solutions in Theorem~\ref{thmproba} is  very different from \cite{Bo}. It is based on a probabilistic energy estimate introduced in \cite{BTz14} (see also \cite{CO}) while the argument giving the globalisation of the local solutions in \cite{Bo} is restricted to a very particular distribution of the initial data.  Finally, Theorem~\ref{thmproba} deals with functions of positive Sobolev regularity which avoids a renormalization of the equation, making the results more natural from a purely PDE perspective.  

Strictly speaking, the result of  Theorem~\ref{thmproba} is not stated as such in \cite{BTz14},\cite{OhPo},\cite{SXia}. One may however adapt the argument presented in \cite{Tzbook} which proves Theorem~\ref{thmproba} for $\sigma=1$ to the case of $\sigma\in [1,2]$. 

The regularization by convolution used in Theorem~\ref{thmproba} is essential. We refer to \cite{Tzbook,Xthese} for results showing that other regularizations of $(f,g)\in\Sigma$ may give divergent sequences of smooth solutions.  

The main result of this paper is that even if we naturally regularize the data by convolution, 
there is a dense set of (pathological) initial data giving not converging smooth solutions.  
This is in some sense a complementary to  Theorem~\ref{thmproba} result.

In order to state our result, we fix a bump function $\rho\in C_c^{\infty}(\R^3)$ such that 
$$
0\leq \rho(x)\leq 1,\quad \rho|_{|x|>\frac{1}{100}}\equiv 0,\quad \int_{\R^3}\rho(x)dx=1.
$$
For any $\epsilon>0$, we define $\rho_{\epsilon}(x):=\epsilon^{-3}\rho(x/\epsilon)$. With this notation, we have the following statement. 
\begin{thm}\label{thm:main}
Let $\frac{1}{2}\leq\sigma\leq 2$ and $\max\{0,\frac{3}{2}-\frac{2}{2\sigma-1} \}< s<s_c=\frac{3}{2}-\frac{1}{\sigma}$. 
There exists a dense set $S\subset \mathcal{H}^s(\T^3)$, such that for every $(f,g)\in S$, the family of global smooth solutions  $(u^{\epsilon})_{t>0}$ of \eqref{wave:main} with initial data $(\rho_{\epsilon}\ast f,  \rho_{\epsilon}\ast g)$ does not converge. More precisely,
\begin{align}\label{limsup}
 \limsup_{\epsilon\rightarrow 0}\|u^{\epsilon}(t)\|_{L^{\infty}([0,1];H^s(\T^3))}=+\infty.
\end{align}
\end{thm}
\begin{rem}
	Furthermore, when $\sigma\geq 1$ (including the cubic nonlinearity), we are able to show that
\begin{align}\label{lim}
\lim_{\epsilon\rightarrow 0}\|u^{\epsilon}(t)\|_{L^{\infty}([0,1];H^s)(\T^3)}=+\infty.
\end{align}
See Remark \ref{rq}.
\end{rem}

It turns out that as a consequence of Theorem \ref{thm:main}, the pathological set
\begin{align*} \mathcal{P}:=\{(f,g)\in\mathcal{H}^s(\T^3):\,\,&\text{the solution } u^{\epsilon}(t) \text{ of \eqref{wave:main} with initial data } \rho_{\epsilon}\ast (f,g),\\ &
\text{satisfies the property }\quad
\limsup_{\epsilon\rightarrow 0} \|u^{\epsilon}(t)\|_{L^{\infty}([0,1];H^s(\T^3))}=+\infty \;  \}
\end{align*}
		contains a dense $G_{\delta}$ set:
\begin{cor}\label{Gdelta}
Under the same condition as Theorem \ref{thm:main}, the pathological data set $\mathcal{P}$ of \eqref{wave:main} such that \eqref{limsup} holds contains a dense $G_{\delta}$ subset of $\mathcal{H}^s(\T^3)$.
\end{cor}
Consequently, by the Baire category theorem, the good data set $\Sigma$ in Theorem \ref{thmproba} cannot be $G_{\delta}$. On the other hand, the pathological set is negligible with respect to the measures used in the probabilistic well-posedness of \eqref{wave:main}. This shows that the topological and the measure theoretic notions of genericity are very different. For examples of $G_{\delta}$ dense sets giving solutions of Hamiltonian PDE's with growing Sobolev norms for large times, we refer to  \cite{GeG1},\cite{GeG2},\cite{Ha}, while in Corollary \ref{Gdelta}, the Sobolev norms are growing in very short times, depending on the frequency localization of the initial data.

The main ingredient of the proof of Theorem \ref{thm:main} is a refined version of the ill-posedness construction in \cite{BTz08} (see also \cite{STz19}) which uses an idea of Lebeau \cite{Le01} exploiting the property of the finite propagation speed of the wave equation. It is an interesting problem to extend the result of Theorem~\ref{thm:main} to the case of the nonlinear Schr\"odinger equation. Such a result would be a significant extension of  \cite{AC}.

The results of  Theorem~\ref{thmproba} and Theorem~\ref{thm:main}  show that for data of supercritical regularity two opposite behaviours  coexiste. Both behaviours are manifested on dense sets which makes that it would be probably interesting to try to observe these behaviours by numerical simulations. 
%%%
\subsection*{Acknowledgement}
The authors are supported by the ANR grant ODA (ANR-18-CE40-0020-01). The authors wish to thank Nicolas Burq and Patrick G\'erard for pointing out that our construction in the previous version of this article implies that the pathological data set contains a dense $G_{\delta}$ set. 
%%%%
\section{Unstable profile}
\subsection{Explicit estimates for the ODE profile}
Let $V(t)$ be the unique solution of the following ODE:
\begin{equation}\label{ODE}
V''+|V|^{2\sigma}V=0,\quad V(0)=1,\; V'(0)=0.
\end{equation}
It can be shown that $V(t)$ is periodic (see Lemma 6.2 of \cite{STz19}). We choose the following parameters:
\begin{equation}\label{parameters}
 \kappa_n=(\log n)^{-\delta_1},\; \epsilon_n=\frac{1}{100n},\; t_n=\big((\log n)^{\delta_2}n^{-\big(\frac{3}{2}-s\big)}\big)^{\sigma},\; \lambda_n=(\kappa_n n^{\frac{3}{2}-s})^{\sigma},
\end{equation}
where $0<\delta_1<\delta_2<1$ and their precise values are to be chosen according to different context.

Take $\varphi\in C_c^{\infty}(|x|\leq 1)$, radial, $0\leq \varphi\leq 1$, and $\nabla\varphi\neq 0$ on $0<|x|<1$. Let
\begin{equation}\label{profile1}
v_n(0,x):=\kappa_n n^{\frac{3}{2}-s}\varphi(nx),\quad v_n^{\epsilon}(0):=\rho_{\epsilon}\ast v_n(0).
\end{equation} 
Define
\begin{equation}\label{profile2}
v_n^{\epsilon}(t,x)=v_n^{\epsilon}(0,x)V(t(v_n^{\epsilon}(0,x))^{\sigma}). 
\end{equation}
Then one verifies that $v_n^{\epsilon}$ solves
\begin{equation}\label{ODE1}
\partial_t^2v_n^{\epsilon}+|v_n^{\epsilon}|^{2\sigma}v_n^{\epsilon}=0,\quad  (v_n^{\epsilon},\partial_tv_n^{\epsilon})|_{t=0}=(v_n^{\epsilon}(0), 0).
\end{equation}
\begin{lem}\label{profilebound}
Let $0\leq s<s_c$, then for parameters defined in \eqref{parameters},
\begin{enumerate}
	\item $\|v_n^{\epsilon_n}(t_n)\|_{H^s(\T^3)}\gtrsim \kappa_n(\lambda_n t_n)^s$.
	\item $\|v_n^{\epsilon_n}(t)\|_{H^{k}(\T^3)}\lesssim \kappa_n(\lambda_n t_n)^{k}n^{k-s}$, for $k=0,1,2,3,\cdots$ and $t\in[0,t_n]$.
	\item $\|v_n^{\epsilon_n}(t)\|_{L^{\infty}(\T^3)}\lesssim \lambda_n^{\frac{1}{\sigma}}$.
	\item $\|\partial^{\alpha}v_n^{\epsilon_n}(t)\|_{L^{\infty}(\T^3)}\lesssim \lambda_n^{\frac{1}{\sigma}}n^{|\alpha|}(1+\lambda_nt),$ for $\alpha\in\N^3, |\alpha|=1$ and $t\in[0,t_n]$.
\end{enumerate}	

%where we use the notation
%\begin{equation*}
%\nabla^k:=\begin{cases}
%& \mathrm{Id}, \;k=0,\\
%& \nabla, \; k=1.
%\end{cases}
%\end{equation*}
\end{lem}
\begin{proof}
The proof follows from a direct calculation as in \cite{BTz08}, with an additional attention to the convolution. We denote by $T_{\lambda}$, the scaling operator $T_{\lambda}(f):=f(\lambda\cdot)$. Without loss of generality, we will do all the computation in $\R^3$ instead of $\T^3$, since all the functions involved are compactly supported near the origin.

By definition, for $\alpha\in\N^3, |\alpha|=k$,
$$ v_n^{\epsilon_n}(0,x)=\lambda_n^{\frac{1}{\sigma}}\int_{\R^3}\varphi(n(x-y))\frac{1}{\epsilon_n^3}\rho\big(\frac{y}{\epsilon_n}\big)dy,\quad \partial^{\alpha} v_n^{\epsilon_n}(0,x)=\lambda_n^{\frac{1}{\sigma}}n^k\int_{\R^d}T_n(\partial^{\alpha}\varphi)(x-y)\frac{1}{\epsilon_n^3}\rho\big(\frac{y}{\epsilon_n}\big)dy.
$$
Using Young's convolution inequality, we have from \eqref{profile2} that
$$ \|\partial^{\alpha}v_n^{\epsilon_n}(0)\|_{L^{\infty}}\lesssim \lambda_n^{\frac{1}{\sigma}}n^{|\alpha|},\quad \|\partial^{\alpha}v_n^{\epsilon_n}(0)\|_{L^2}\lesssim \kappa_n n^{|\alpha|-s},\quad  \|v_n^{\epsilon_n}(t)\|_{L^{\infty}}\lesssim \lambda_n^{\frac{1}{\sigma}},
$$
and
$$ \|v_n^{\epsilon_n}(t)\|_{L^2}\leq \|V\|_{L^{\infty}}\|v_n^{\epsilon_n}(0)\|_{L^2}\lesssim \kappa_n n^{-s}.
$$
This proves (2) and (3) for the case $k=0$. From direct calculation using \eqref{profile2},
\begin{equation}\label{calculation1}
\begin{split}
\nabla v_n^{\epsilon_n}(t,x)=&\sigma t(v_n^{\epsilon_n}(0,x))^{\sigma}\nabla v_n^{\epsilon_n}(0,x)V'\big(t(v_n^{\epsilon_n}(0,x) )^{\sigma} \big)+\nabla v_n^{\epsilon_n}(0,x)V\big(t(v_n^{\epsilon_n}(0,x))^{\sigma} \big).
\end{split}
\end{equation}
Thus
$ \|\nabla v_n^{\epsilon_n}(t)\|_{L^{\infty}}\lesssim (\lambda_n t+1)\lambda_n^{\frac{1}{\sigma}}n.
$
Note that $\lambda_nt_n=(\log n)^{\sigma(\delta_2-\delta_1)}\gg 1$, the dominant part in $\partial^{\alpha}v_n^{\epsilon_n}(t,x)$ comes from
$$ \big((v_n^{\epsilon_n}(0))^{\sigma-1}\nabla v_n^{\epsilon_n}(0)\big)^{|\alpha|}t^{|\alpha|}v_n^{\epsilon_n}(0)V^{(|\alpha|)}(\cdot),
$$
if we estimate $t$ by $t_n$, hence $\|v_n^{\epsilon_n}(t)\|_{H^k}\lesssim \kappa_n(\lambda_nt_n)^k n^{k-s}$, for all $k=0,1,2,\cdots$. This proves (2). 

The only non-trivial part is (1). Since $0<s<1$, from the  interpolation
$$ \|v_n^{\epsilon_n}(t)\|_{H^1}\lesssim \|v_n^{\epsilon_n}(t)\|_{H^s}^{\frac{1}{2-s}}\|v_n^{\epsilon_n}(t)\|_{H^2}^{\frac{1-s}{2-s}}
$$ 
and the upper bound of $\|v_n^{\epsilon_n}(t)\|_{H^2}$ that we have proved, it suffices to show that
\begin{equation}\label{H1lower}
\|v_n^{\epsilon_n}(t_n)\|_{H^1}\gtrsim \kappa_n(\lambda_nt_n)n^{1-s}.
\end{equation}
It is reduced to get a lower bound for the dominant part
\begin{equation}\label{dominant1}
\begin{split} &\big\|\sigma t_n\big(v_n^{\epsilon_n}(0,x)\big)^{\sigma}\nabla v_n^{\epsilon_n}(0,x)V'\big(t_n(v_n^{\epsilon_n}(0,x))^{\sigma} \big)\big\|_{L^2}\\
=&\sigma t_nn\lambda_n^{1+\frac{1}{\sigma}}\big\| \big[ (T_n(\nabla\varphi))\ast \rho_{\epsilon_n} \big] 
\big[ (T_n(\varphi))\ast \rho_{\epsilon_n} \big]^{\sigma}   V'\big(\lambda_n t_n ((T_n\varphi)\ast\rho_{\epsilon_n})^{\sigma} \big)\big\|_{L^2}
\end{split}
\end{equation}
Note that $(T_nf)\ast \rho_{\epsilon_n}(x)=\int f(nx-n\epsilon_n y)\rho(y)dy$, hence
$$ (\text{RHS. of }  \eqref{dominant1})\sim t_nn^{1-\frac{3}{2}}\lambda_n^{1+\frac{1}{\sigma}}
\big\|\nabla(\varphi\ast \widetilde{\rho})\cdot (\varphi\ast \widetilde{\rho} )^{\sigma}V'\big(\lambda_nt_n(n\epsilon_n)^{-3\sigma }(\varphi\ast \widetilde{\rho})^{\sigma}(x) \big) \big\|_{L^2},
$$
where $\widetilde{\rho}=T_{\frac{1}{n\epsilon_n}}\rho=T_{100}\rho$ and we used $n\epsilon_n=\frac{1}{100}$.
Note that $t_nn^{1-\frac{3}{2}}\lambda_n^{1+\frac{1}{\sigma}}=\lambda_n t_nn^{1-s}$, hence \eqref{H1lower} follows from the following lemma:
\begin{lem}\label{lowerboundgeometric}
Assume that $\psi\in C_c^{\infty}(\R^d)$ and $\psi(x)>0$ for all $|x|<1$. Assume that there exist two constants $0<a<b<1$, such that $d\psi\neq 0$ on $\{x:a\leq |x|\leq b\}$. Let $W$ be a non-trivial periodic function (i.e. $W\neq 0$). Then there exist $c_0>0, \lambda_0>0$, such that for all $\lambda\geq \lambda_0$,
$$ \big\|\nabla\psi(x)|\psi(x)|^{\sigma}W(\lambda\psi(x)) \big\|_{L^2(\R^d)}\geq c_0>0.
$$	
\end{lem}
\begin{proof}
We follow the geometric argument in \cite{STz19}. Denote by $\mathcal{C}_{a,b}:=\{x:a\leq |x|\leq b \}$. By shrinking $a,b$ if necessary, we may assume that $\psi(\mathcal{C}_{a,b})$ is foliated by $\Sigma_s:=\{x:\psi(x)=s \}$. From the hypothesis on $\psi$, there exist $0<c_1<C_1<\infty$, such that $c_1\leq |\nabla\psi|\leq C_1$ on $\mathcal{C}_{a,b}$. Let $B=\max_{\mathcal{C}_{a,b}}\psi$ and $A=\min_{\mathcal{C}_{a,b}}\psi$, then we have for $F(s)=|s|^{2\sigma}|W(\lambda s)|^2$ that
$$ \|\nabla\psi (F\circ\psi)^{1/2}\|_{L^2}^2\geq c_1^2\int_{\mathcal{C}_{a,b}}F(\psi(x))dx.
$$ 
By the co-area formula, 
$$ \int_{\mathcal{C}_{a,b}}F(\psi(x))dx=\int_A^BF(s)ds\int_{\Sigma_s}\frac{d\sigma_{\Sigma_s}}{|\nabla\psi|}\geq c'\int_A^B|s|^{2\sigma}|W(\lambda s)|^2ds,
$$
thanks to the fact that the mapping $s\mapsto \mathcal{M}^{d-1}(\Sigma_s)$ is continuous, where $\mathcal{M}^{d-1}$ is the surface measure on $\Sigma_s$. By changing variables, we obtain that
$$ \int_A^B|s|^{2\sigma}|W(\lambda s)|^2ds=\frac{1}{\lambda^{2\sigma+1}}\int_{\lambda A}^{\lambda B}|s|^{2\sigma}|W(s)|^2ds\geq C_{A,B}\frac{1}{\lambda(B-A)}\int_{\lambda A}^{\lambda B}|W(s)|^2ds\geq C_{A,B}',
$$
where the last constant does not depend on $\lambda$, if $\lambda$ is large enough. This completes the proof of Lemma~\ref{lowerboundgeometric}.
\end{proof}

The proof of Lemma \ref{profilebound} is now complete.
\end{proof}

\begin{rem}\label{rq}
When $\sigma\geq 1$, the statements of Lemma \ref{profilebound} hold for all $\epsilon\leq \epsilon_n^{2}$. 
Indeed, all the inequalities hold automatically for $\epsilon\leq\epsilon_n^2$, except for (1), the lower bound of $\|v_n^{\epsilon}(t_n)\|_{H^s(\T^3)}$. To get (1), it suffices to prove \eqref{H1lower} when replacing $\epsilon_n$ by $\epsilon\leq \epsilon_n^2$.
It is then reduced to show that
\begin{align}\label{reduceddifference} 
&\|(T_n\nabla\varphi) (T_n\varphi)^{\sigma}V'(\lambda_nt_n(T_n\varphi )^{\sigma})-
(T_n\nabla\varphi)\ast\rho_{\epsilon}\cdot (T_n\varphi\ast\rho_{\epsilon})^{\sigma}V'(\lambda_nt_n(T_n\varphi\ast\rho_{\epsilon} )^{\sigma}) \|_{L^2}\notag \\ 
\leq & o(n^{-\frac{3}{2}}), 
\end{align}
as $n\rightarrow\infty$. 
Note that $$\|T_n\nabla\ast \rho_{\epsilon}\|_{L^{\infty}}+\|T_n\varphi\ast\rho_{\epsilon}\|_{L^{\infty}}\lesssim 1, \; \|V^{(k)}(\lambda_n t_n(T_n\varphi\ast\varphi_{\epsilon} ))\|_{L^{\infty}}\lesssim 1, $$
for $k=1,2$, with constants independent of $\epsilon$. By taking the Fourier transform, for any Schwartz function $F$, uniformly in $\epsilon\leq \epsilon_n^2$, we have
$$ \|T_n F-T_nF\ast \rho_{\epsilon}\|_{L^2}\leq n^{-\frac{3}{2}}\|(\widehat{\rho}(n\epsilon\xi)-\widehat{\rho}(0))\widehat{F}(\xi)\|_{L^2}=o(n^{-\frac{3}{2}}),\; n\rightarrow\infty,
$$
thanks to the dominated convergence theorem. Together with the fact $\sigma\geq 1$ and the mean value theorem, we obtain \eqref{reduceddifference}.
The above argument, combed with slight modifications of the analysis below, allows us to prove \eqref{lim}.
\end{rem}

\subsection{Perturbative analysis}
Fix $(u_0,u_1)\in C^{\infty}(\T^3)\times C^{\infty}(\T^3)$, denote by $u_n^{\epsilon_n}$ the solution of
$$ \partial_t^2u_n^{\epsilon_n}-\Delta u_n^{\epsilon_n}+|u_n^{\epsilon_n}|^{2\sigma}u_n^{\epsilon_n}=0
$$
with the initial data $(u_n^{\epsilon_n}(0),\partial_tu_n^{\epsilon_n}(0))=\rho_{\epsilon_n}\ast\big((u_0,u_1)+(v_n(0),0) \big)$, where $v_n(0)$ is given by \eqref{profile1}. We denote  by 
$$ S(t)(f,g):=\cos(t\sqrt{-\Delta})f+\frac{\sin\sqrt{-\Delta}}{\sqrt{-\Delta}}g
$$
the propagator of the linear wave equation.
\begin{prop}\label{pertubation}
 Assume that $\max\big\{\frac{3}{2}-\frac{2}{2\sigma-1}, 0\big\}\leq s<s_c=\frac{3}{2}-\frac{1}{\sigma}$, then for any $0<\theta<\frac{\sigma}{2}\big(\frac{3}{2}-s\big)-\frac{1}{2}$ and $(u_0,u_1)\in C^{\infty}(\T^3)\times C^{\infty}(\T^3)$, there exist $C>0$, $\delta_2>0$, such that for any $\delta_1\in(0,\delta_2)$, we have
$$ \sup_{t\in[0,t_n]}\|u_n^{\epsilon_n}(t)-\rho_{\epsilon_n}\ast S(t)(u_0,u_1)-v_n^{\epsilon_n}(t) \|_{H^{\nu}(\T^3)}\leq Cn^{(\nu-s)-\theta},\forall \nu=0,1,2,
$$
where the function $v_n^{\epsilon_n}(t)$ is defined in \eqref{profile2} with parameters as  in \eqref{parameters}, and the constant $C$ only depends on the smooth data $(u_0,u_1)$ and $\theta>0$. Consequently, we have
$$ \sup_{t\in[0,t_n]}\|u_n^{\epsilon_n}(t)-\rho_{\epsilon_n}\ast S(t)(u_0,u_1)-v_n^{\epsilon_n}(t) \|_{H^{s}(\T^3)}\leq Cn^{-\theta}.
$$
In particular, for $\delta_1$ sufficiently small, 
$$ \|u_n^{\epsilon_n}(t_n)\|_{H^s(\T^3)}\gtrsim (\log n)^{s\sigma(\delta_2-\delta_1)-\delta_1}\rightarrow\infty,\text{ as }n\rightarrow \infty.
$$
\end{prop}
\begin{proof}
	Denote by $u_L^{\epsilon_n}(t)=\rho_{\epsilon_n}\ast S(t)(u_0,u_1)$ the linear solution with regularized data $\rho_{\epsilon_n}\ast (u_0,u_1)$\footnote{Since we work on $\T^3$, the convolution $\rho_{\epsilon}$ commutes with free propagators $\cos(t\sqrt{-\Delta})$ and $\frac{\sin(t\sqrt{-\Delta})}{\sqrt{-\Delta}}$. For the wave equation on general manifolds, one should take the linear solution as $S(t)(\rho_{\epsilon_n}\ast(u_0,u_1))$. }. Then for $k=0,1,2,3$,
	\begin{align}\label{boundlinear} 
 \|\nabla^ku_L^{\epsilon_n}(t)\|_{L^{\infty}(\T^3)}\lesssim 1,
		\end{align}
	uniformly in $n$, where the implicit constant depends only on finitely many norms of the smooth linear solution $S(t)(u_0,u_1)$.
	
	Denote by $f(v)=|v|^{2\sigma}v$. Consider the difference $w_n=u_n^{\epsilon_n}-u_L^{\epsilon_n}-v_n^{\epsilon_n}$, it satisfies the equation
	$$ \partial_t^2w_n-\Delta w_n=\Delta v_n^{\epsilon_n}-\big(f(v_n^{\epsilon_n}+u_L^{\epsilon_n}+w_n)-f(v_n^{\epsilon_n}) \big),\quad (w_n,\partial_tw_n)|_{t=0}=0.
	$$
	Define the semi-classical energy for $w_n$ as in \cite{BTz08} 
	\begin{equation}\label{energy:semi-classic}
	\begin{split}
	E_n(t):=&\frac{1}{n^{2(1-s)}}\big(\|\partial_tw_n(t)\|_{L^2(\T^3)}^2+\|\nabla w_n(t)\|_{L^2(\T^3)}^2 \big)\\
	+&\frac{1}{n^{2(2-s)}}\big(\|\partial_tw_n(t)\|_{H^1(\T^3)}^2+\|\nabla w_n(t)\|_{H^1(\T^3)}^2 \big).
	\end{split}
	\end{equation}
	Here the second line in \eqref{energy:semi-classic} is needed since we need to use it to control the $L^{\infty}$ norm of $w_n$.
	
	Let $F_n(t)=-\Delta v_n^{\epsilon_n}+f(v_n^{\epsilon_n}+u_L+w_n)-f(v_n^{\epsilon_n})$. From the energy estimate for the inhomogeneous linear wave equation, we have
	\begin{align*}
	\frac{1}{2}\frac{d}{dt}E_n(t)\leq &Cn^{-(1-s)} \|n^{-(1-s)}\partial_tw_n(t)\|_{L^2(\T^3)}\|F_n(t)\|_{L^2(\T^3)}\\+&Cn^{-(2-s)}\|n^{-(2-s)}\partial_tw_n(t)\|_{H^1(\T^3)}\|F_n(t)\|_{H^1(\T^3)},
	\end{align*}
	and this implies that
	\begin{align}\label{energybound}
 \frac{d}{dt}(E_n(t))^{1/2}\leq C\big(n^{-(1-s)}\|F_n(t)\|_{L^2(\T^3)}+n^{-(2-s)}\|F_n(t)\|_{H^1(\T^3)} \big).
		\end{align}
To simplify the notation, we denote by
$$ e_n(t):=\sup_{0\leq\tau\leq t}\big(E_n(t)\big)^{\frac{1}{2}}.
$$
Our goal is to show that $\sup_{t\in[0,t_n]}e_n(t)\lesssim n^{-\theta}$.
Write
$$ G_n(t):=f(v_n^{\epsilon_n}+u_L^{\epsilon_n}+w_n)-f(v_n^{\epsilon_n}),
$$
from Lemma \ref{profilebound}, we have, for $t\in[0,t_n]$ that
\begin{align}\label{Fnbound}
 \|F_n(t)\|_{L^2(\T^3)}\lesssim \kappa_n(\lambda_nt_n)^2n^{2-s}+\|G_n(t)\|_{L^2(\T^3)}. 
\end{align}
By the Taylor expansion,
$$ |G_n|\lesssim (|u_L^{\epsilon_n}|+|w_n|)(|v_n^{\epsilon_n}|^{2\sigma}+|u_L^{\epsilon_n}|^{2\sigma}+|w_n|^{2\sigma}), 
$$
hence
$$ \|G_n(t)\|_{L^2(\T^3)}\lesssim \|w_n(t)\|_{L^2(\T^3)}\big(1+\|v_n^{\epsilon_n}(t)\|_{L^{\infty}(\T^3)}^{2\sigma}+\|w_n(t)\|_{L^{\infty}(\T^3)}^{2\sigma}\big)+\|v_n^{\epsilon_n}(t)\|_{L^2(\T^3)}\|v_n^{\epsilon_n}(t)\|_{L^{\infty}(\T^3)}^{2\sigma-1},
$$
where we used \eqref{boundlinear}.
By writing $w_n(t,x)=\int_0^t\partial_tw_n(\tau,x)d\tau$ (since $w_n(0,\cdot)=0$), we obtain that
\begin{align}\label{Gnbound}
\|G_n(t)\|_{L^2(\T^3)}\lesssim &\int_0^t\|\partial_tw_n(\tau)\|_{L^2(\T^3)}d\tau\cdot\big(1+\|v_n^{\epsilon_n}(t)\|_{L^{\infty}(\T^3)}^{2\sigma}+\|w_n(t)\|_{L^{\infty}(\T^3)}^{2\sigma}\big) \notag\\
+&\|v_n^{\epsilon_n}(t)\|_{L^2(\T^3)}\|v_n^{\epsilon_n}(t)\|_{L^{\infty}(\T^3)}^{2\sigma-1}+1 \notag\\
\lesssim &tn^{1-s}e_n(t)(\lambda_n^2+\|w_n(t)\|_{L^{\infty}(\T^3)}^{2\sigma})+\kappa_n\lambda_n^{2-\frac{1}{\sigma}}n^{-s},
\end{align}
where we have used Lemma \ref{profilebound} to control $\|v_n^{\epsilon_n}(t)\|_{L^{\infty}}$.
Similarly, for $t\in[0,t_n]$, we have
\begin{align}\label{Fnbound1}
\|\nabla F_n(t)\|_{L^2(\T^3)}\lesssim \kappa_n(\lambda_nt_n)^3n^{3-s}+\|\nabla G_n(t)\|_{L^2(\T^3)}.
\end{align}
We need to estimate $\|w_n(t)\|_{L^{\infty}(\T^3)}$. From the Gagliardo-Nirenberg inequality,
\begin{align}\label{Linfty}
\|w_n(t)\|_{L^{\infty}(\T^3)}\lesssim \|w_n(t)\|_{H^2(\T^3)}^{\frac{3}{4}}\|w_n(t)\|_{L^2(\T^3)}^{\frac{1}{4}}\lesssim (n^{2-s}e_n(t))^{\frac{3}{4}}(te_n(t)n^{1-s})^{\frac{1}{4}}=t^{\frac{1}{4}}n^{\frac{7}{4}-s}e_n(t),
\end{align}
where we used $w_n(t)=\int_0^t\partial_tw_n(\tau,\cdot)d\tau$ again. Since $t\leq t_n=(\log n)^{\sigma\delta_2}n^{-\big(\frac{3}{2}-s\big)\sigma}$ and $\sigma\big(\frac{3}{2}-s\big)>1$, we have
\begin{align}\label{boundLinfty}
\|w_n(t)\|_{L^{\infty}(\T^3)}\lesssim n^{\frac{3}{2}-s}e_n(t).
\end{align}
Therefore,
\begin{align*}
n^{-(1-s)}\|F_n(t)\|_{L^2(\T^3)}\lesssim & \kappa_n(\lambda_nt_n)^2n+\kappa_n(\kappa_nn^{\frac{3}{2}-s })^{2\sigma-1}n^{-1}+t_ne_n(t)\big((\kappa_n n^{\frac{3}{2}-s})^{2\sigma}+(n^{\frac{3}{2}-s}e_n(t))^{2\sigma} \big)\\
\lesssim &(\log n)^{2\sigma(\delta_2-\delta_1)-\delta_1}n+(\log n)^{-2\sigma\delta_1}n^{(2\sigma-1)\big(\frac{3}{2}-s\big)-1}\\
+&n^{\big(\frac{3}{2}-s\big)\sigma}e_n(t)\big[(\log n)^{\sigma(\delta_2-2\delta_1)}+(\log n)^{\sigma\delta_2}(e_n(t))^{2\sigma} \big].
\end{align*}
Since $s>\frac{3}{2}-\frac{2}{2\sigma-1}$, we have $(2\sigma-1)\big(\frac{3}{2}-s\big)-1<1$, thus
\begin{align}\label{Fnbound2}
n^{-(1-s)}\|F_n(t)\|_{L^2(\T^3)}\lesssim (\log n)^{\sigma(2\delta_2-3\delta_1)}n+(\log n)^{\sigma\delta_2}n^{\big(\frac{3}{2}-s\big)\sigma}e_n(t)(1+(e_n(t))^{2\sigma}).
\end{align}

Next we estimate $|\nabla G_n|$ as
\begin{align*}
|\nabla G_n|\lesssim & |\nabla v_n^{\epsilon_n}|\big(1+|v_n^{\epsilon_n}|^{2\sigma-1}+|w_n|^{2\sigma-1} \big)\big(1+|w_n|\big)\\
+&\big(1+|v_n^{\epsilon_n}|^{2\sigma}+|w_n|^{2\sigma}\big)\big(1+|\nabla w_n| \big),
\end{align*}
where the implicit constants are independent of $n$, thanks to \eqref{boundlinear}. To estimate the $L^2$ norm of $\nabla G_n$, we organize the terms as
$$ \|\nabla v_n^{\epsilon_n}(1+|v_n^{\epsilon_n}|^{2\sigma-1}+|w_n|^{2\sigma-1})w_n \|_{L^2}\leq \|w_n\|_{L^2}\|\nabla v_n^{\epsilon_n}\|_{L^{\infty}}\big(1+\|v_n^{\epsilon_n}\|_{L^{\infty}}^{2\sigma-1}+\|w_n\|_{L^{\infty}}^{2\sigma-1}  \big),
$$
$$ \| (1+|v_n^{\epsilon_n}|^{2\sigma}+|w_n|^{2\sigma})\nabla w_n \|_{L^2}\leq \|\nabla w_n\|_{L^2}\big(1+\|v_n^{\epsilon_n}\|_{L^{\infty}}^{2\sigma}+\|w_n\|_{L^{\infty}}^{2\sigma}  \big),
$$
$$ \|\nabla v_n^{\epsilon_n}(1+|v_n^{\epsilon_n}|^{2\sigma-1}+|w_n|^{2\sigma-1})\|_{L^2}\leq \|\nabla v_n^{\epsilon_n}\|_{L^2}\big(1+\|v_n^{\epsilon_n}\|_{L^{\infty}}^{2\sigma-1}+\|w_n\|_{L^{\infty}}^{2\sigma-1}  \big),
$$
$$  \|(1+|v_n^{\epsilon_n}|^{2\sigma}+|w_n|^{2\sigma})\|_{L^2}\leq \big(1+\|v_n^{\epsilon_n}\|_{L^{\infty}}^{2\sigma-1}\|v_n^{\epsilon_n}\|_{L^2}+\|w_n\|_{L^{\infty}}^{2\sigma-1}\|w_n\|_{L^2}  \big).
$$
Putting them together and using
\begin{equation}\label{H^k} \|w_n(t)\|_{H^k(\T^3)}=\Big\|\int_0^t\partial_tw_n(\tau)d\tau\Big\|_{H^k(\T^3)}\leq n^{1+k-s}te_n(t),\quad k=0,1,
\end{equation}
we have
\begin{equation}\label{boundGn1}
\begin{split}
n^{-(2-s)}\|\nabla G_n(t)\|_{L^2(\T^3)}\lesssim & (\log n)^{\sigma\delta_2}n^{\big(\frac{3}{2}-s\big)\sigma}e_n(t)\big(1+(e_n(t))^{2\sigma} \big)\\
+&(\log n)^{\sigma(\delta_2-\delta_1)}n^{(2\sigma-1)\big(\frac{3}{2}-s\big)-1}\big(1+(e_n(t))^{2\sigma-1} \big)\\
\lesssim &(\log n)^{\sigma\delta_2}n^{\big(\frac{3}{2}-s\big)\sigma}e_n(t)\big(1+(e_n(t))^{2\sigma} \big)
+(\log n)^{\sigma\delta_2}n(1+e_n(t)^{2\sigma-1} ).
\end{split}
\end{equation}
We observe that
$$ \frac{de_n}{dt}\leq \Big|\frac{d}{dt}(E_n(t))^{1/2}\Big|.
$$
Therefore,
\begin{equation}\label{enboundfinal}
\begin{split}
\frac{de_n}{dt}\leq (\log n)^{3\sigma\delta_2}n+(\log n)^{\sigma\delta_2}n^{\sigma\big(\frac{3}{2}-s\big)}e_n(t)\big(1+(e_n(t))^{2\sigma}\big).
\end{split}
\end{equation}
By the Grownwall type argument, we obtain 
$$ e_n(t)\leq n^{1-\sigma\big(\frac{3}{2}-s\big)}(\log n)^{3\sigma\delta_2}e^{(\log n)^{2\sigma\delta_2}},\quad \forall t\in [0,t_n].
$$
Since $1<\sigma\big(\frac{3}{2}-s\big)$, for any $0<\theta<\frac{\sigma}{2}\big(\frac{3}{2}-s\big)-\frac{1}{2}$, we can choose $\delta_2>0$ sufficiently small, such that the right hand side is smaller than $n^{-\theta}$. Consequently, from \eqref{H^k}, 
$$ \|w_n(t)\|_{L^2(\T^3)}\leq n^{1-s}e_n(t)t\lesssim n^{1-s-\big(\frac{3}{2}-s\big)\sigma}(\log n)^{\delta_2\sigma}n^{-\theta}\lesssim n^{-s-\theta},\;\forall t\leq t_n.
$$
 Finally, the bound for the $H^s$ norm of $w_n(t)$ follows from the interpolation. This completes the proof of Proposition \ref{pertubation}. 
\end{proof}

\section{Proof of the main theorem}
First we recall the following property of finite propagation speed for the wave equation. 
\begin{lem}\label{FPS}
Let $w_1, w_2$ be two $C^{\infty}$ solutions of the nonlinear wave equation
$$ \partial_t^2w-\Delta w+|w|^{2\sigma}w=0.
$$ 	
If the initial data $(w_1(0),\partial_tw_1(0)), (w_2(0),\partial_tw_2(0))$ coincide on the ball $B(x_0,r_0)\subset \R^d$, then
for $0\leq t<r_0$, $(w_1(t),\partial_tw_1(t))=(w_2(t),\partial_tw_2(t) )$ on $B(x_0,r_0-t)$.
\end{lem}
\begin{proof}
Without loss of generality, we may assume that $x_0=0$. Take the difference $u=w_1-w_2$, then
$$ \partial_t^2u-\Delta u+u=V(t,x)u,
$$
where 
$$
V(t,x)=(2\sigma+1)\int_0^1|(1-\lambda)w_1(t,x)+\lambda w_2(t,x)|^{2\sigma}d\lambda+1\in L_{\loc}^{\infty}.
$$ 
For $0\leq t_1<t_2<r_0$, denote by  $\mathcal{C}_{t_1,t_2}(r_0):=\{(t,x): t_1\leq t\leq t_2, |x|\leq r_0-t \}.$ Define the local energy density
$$ e(t,x):=\frac{1}{2}(|\nabla u(t,x)|^2+|\partial_tu(t,x)|^2+|u(t,x)|^2).
$$
Then a direct calculation yields
\begin{equation*}
\begin{split}
\int_{\mathcal{C}_{0,t_0}(r_0)}\partial_tu(\partial_t^2-\Delta+1)udxdt=&\int_0^{t_0}\int_{|x|\leq r_0-t}\frac{d}{dt}e(t,x)dxdt-\int_0^{t_0}\int_{|x|=r_0-t}\partial_t u\partial_rud\sigma(x)dt,
\end{split}
\end{equation*}
where $\partial_ru=\frac{x}{|x|}\cdot\nabla u$ and $r=|x|$. Notice that $\frac{d}{dt}\mathbf{1}_{|x|\leq r_0-t}=-\delta_{|x|=r_0-t}$, we have
\begin{align*}
\int_{\mathcal{C}_{0,t_0}(r_0)}\partial_tu (\partial_t^2-\Delta+1)udxdt=&\Big[\int_{|x|\leq r_0-t}e(t,x)dx\Big]_{t=0}^{t=t_0}+\int_0^{t_0}\int_{|x|=r_0-t}\frac{1}{2}\big[|\partial_tu-\partial_ru|^2+|u|^2\big]d\sigma(x)dt\\
\geq & \Big[\int_{|x|\leq r_0-t}e(t,x)dx\Big]_{t=0}^{t=t_0}.
\end{align*}
Using the equation $\partial_t^2u-\Delta u+u=Vu$,  we have
$$ E(t_0)\leq E(0)+\Big|\int_{\mathcal{C}_{0,t_0}(r_0)}Vu\cdot \partial_t udxdt\Big|\leq E(0)+\|V\|_{L^{\infty}([0,r_0]\times B(0;r_0) )}\int_0^{t_0}E(t)dt,
$$
for all $0\leq t_0<r_0$, where $E(t)=\int_{|x|\leq r_0-t}e(t,x)dx$ is the local energy.
Since $E(0)=0$, from Gronwall's inequality, we deduce that $E(t)\equiv 0$ for all $0\leq t<r_0$. This completes the proof of Lemma \ref{FPS}.
\end{proof}
To prove Theorem~\ref{thm:main}, we need to do some preparations. 
We use the coordinate system $x=(x_1,x')$ near the origin. Let $z^k=(z^k_1,0)$ with $z^k_1=\frac{1}{k}$. Let $n_k=\mathrm{e}^{\mathrm{e}^k}$, and define
 $$v_{0,k}(x):=(\log n_k)^{-\delta_1} n_k^{\frac{3}{2}-s}\varphi(n_k(x_1-z^k_{1}),n_kx')=v_{n_k}(0,\cdot-z^k),$$ where $v_n(0)$ is the initial data of the ill-posed profile defined in \eqref{profile1}. Note that there exists $k_0$, such that for all $k\geq k_0$, the supports of $v_{0,k}$ are pairwise disjoint. Moreover, 
for $k_0\leq k_1<k_2$, $$ \mathrm{dist}\big(\text{supp}(v_{0,k_1}), \text{supp}(v_{0,k_2})\big)\sim \frac{1}{k_1}-\frac{1}{k_2}.
$$
Denote by $B_k=B(z^k,r_k)$, where $r_k=\frac{1}{k^3}$. With sufficiently large $k_0$, the balls $B_k,k\geq k_0$ are mutually disjoint. Moreover, $\text{supp}(\rho_{\epsilon_{n_k}}\ast v_{0,k})\subset B_k$ (recall that $\epsilon_{n_k}=\frac{1}{100n_k}$). 
Another simple observation is that
$$ \text{dist}\big(\text{supp}(\rho_{\epsilon_{n_k}}\ast(v_0-v_{0,k})), B_k \big)\gtrsim \frac{1}{k^2},
$$
where $$v_0=\sum_{k\geq k_0}v_{0,k}\in H^s(\T^3).$$ 
In particular, for any $(f,g)\in C^{\infty}\times C^{\infty}$, $\rho_{\epsilon_{n_k}}\ast((f,g)+(v_{0},0) )$ coincides with $\rho_{\epsilon_{n_k}}\ast((f,g)+(v_{0,k},0) )$ on $B_k$. Let $\widetilde{B}_k=B(z^k,r_k/3)$ be a slightly smaller ball. We observe that for $k$ large enough,
$$ \text{supp}(\rho_{\epsilon_{n_k}}\ast v_{0,k} )\subset \widetilde{B}_k.
$$
%\begin{lem}\label{Sobolevrestriction}
%Let $\Omega_1,\Omega_2$ be two connected open subsets of $\R^d$, such that $\ov{\Omega}_1\cap\ov{\Omega}_2=\emptyset.$ Denote by $\Omega=\Omega_1\cup\Omega_2$. Then for any function $u\in H^{\nu}(\Omega)$ and $\nu\in\N$, we have
%$$ \|u\|_{H^{\nu}(\Omega_j)}\leq \|u\|_{H^{\nu}(\Omega)},\quad j=1,2.
%\end{lem}
%\begin{proof}
%	The case $\nu=0$ is trivial.
%\end{proof}
Now we are able to prove Theorem \ref{thm:main}. 
\begin{proof}[Proof of Theorem \ref{thm:main}]
Define
$$ S=C^{\infty}(\T^3)\times C^{\infty}(\T^3)+\Big\{\Big(\sum_{k=k_1}^{\infty}v_{0,k},0\Big): k_1\geq k_0 \Big\}.
$$
Using
$$ \Big\|\sum_{k=k_1}^{\infty}v_{0,k}\Big\|_{H^s(\T^3)}\leq \sum_{k=k_1}^{\infty}\|v_{0,k}\|_{H^s(\T^3)}\leq \sum_{k=k_1}^{\infty}\mathrm{e}^{-k\delta_1}\rightarrow 0 \text{ as }k_1\rightarrow\infty,
$$
we deduce $S$ is dense in $\mathcal{H}^s(\T^3)$. Now fix $(f,g)\in S$. Then by definition, there exists $(u_0,u_1)\in C^{\infty}\times C^{\infty}$ and $k_1\geq k_0$, such that
$$ (f,g)=(u_0,u_1)+\Big(\sum_{k=k_1}^{\infty}v_{0,k},0\Big).
$$
Our goal is to show that, for any $N>0$ and any $\delta>0$, there exist $\tau_N\in[0,1]$ and $0<\epsilon<\delta$, such that the solution $u^{\epsilon}$ to \eqref{wave:main} with initial data $\rho_{\epsilon}\ast(f,g)$ satisfies
\begin{equation}\label{blowup}
\|u^{\epsilon}(\tau_N)\|_{H^s(\T^3)}>N.
\end{equation}
First we choose $k\geq k_1$, large enough, such that 
$$ \kappa_{n_k}(\lambda_{n_k}t_{n_k})^s>N,\quad  \epsilon_k=\frac{1}{100n_k}<\delta.
$$
This can be achieved by choosing $\delta_1<\delta_2$ such that 
$s\sigma(\delta_2-\delta_1)>\delta_1$. Recall that the parameters $\kappa_{n_k}=\mathrm{e}^{-k\delta_1}, \lambda_{n_k}t_{n_k}=\mathrm{e}^{(\delta_2-\delta_1)k\sigma}$ are given by \eqref{parameters}.
 Let $\widetilde{u}_k$ be the solution of \eqref{wave:main} with the initial data $\rho_{\epsilon_{n_k}}\ast (u_0,u_1)+\rho_{\epsilon_{n_k}}\ast(v_{0,k},0)$. Let $\widetilde{v}_k$ be the solution of 
$ \partial_t^2\widetilde{v_k}+|\widetilde{v_k}|^{2\sigma}\widetilde{v_k}=0
$ 
with the initial data $\rho_{\epsilon_{n_k}}\ast (v_{0,k},0)$. We remark that $\widetilde{v}_k,\widetilde{u}_k$ are just $v_{n_k}^{\epsilon_{n_k}},u_{n_k}^{\epsilon_{n_k}}$ in Proposition \ref{pertubation} up to translation. In particular,
\begin{equation}\label{lowerbound1}
\|\widetilde{u}_k(t_{n_k})\|_{H^s(\T^3)}\gtrsim (\log n_k)^{s\sigma(\delta_2-\delta_1)-\delta_1},
\end{equation}
and
\begin{equation}\label{approximation}
\|\widetilde{u}_k(t_{n_k})-\rho_{\epsilon_{n_k}}\ast S(t_{n_k})(u_0,u_1)-\widetilde{v}_k(t_{n_k})\|_{H^s(\T^3)}\lesssim n_k^{-\theta}.
\end{equation}

We have that supp$(\widetilde{v}_k(t))\subset \widetilde{B}_k$ for all $t\in[0,t_{n_k}]$. Now we apply Lemma \ref{FPS} to $\widetilde{u}_k$ and $u^{\epsilon_{n_k}}$. Since at $t=0$, 
$ (u^{\epsilon_{n_k}}(0),\partial_tu^{\epsilon_{n_k}}(0))|_{B_k}=(\widetilde{u}_k(0),\partial_t\widetilde{u}_k(0) )|_{B_k},
$
we deduce that
$$ (u^{\epsilon_{n_k}}(t),\partial_tu^{\epsilon_{n_k}}(t))|_{B(z^k,r_k-t)}=(\widetilde{u}_k(t),\partial_t\widetilde{u}_k(t) )|_{B(z^k,r_k-t)},\quad \forall 0\leq t< r_k.
$$
In particular, for large $k$,
\begin{equation}\label{consequenceFPS} (u^{\epsilon_{n_k}}(t),\partial_tu^{\epsilon_{n_k}}(t))|_{B(z^k,r_k/2)}=(\widetilde{u}_k(t),\partial_t\widetilde{u}_k(t) )|_{B(z^k,r_k/2)},\quad \forall t\in[0,t_{n_k}].
\end{equation}
\begin{lem}\label{restriction}
Assume that $s_1\geq 0$. Let $u\in H^{s_1}(\T^3)$ and $\chi\in C_c^{\infty}(\T^3)$. Then there exists $A>0$, depending only on the function $\chi$ and $s_1$, such that for any $R\geq 1$ 
$$ \|(1-\chi(Rx))u\|_{H^{s_1}(\T^3)}+\|\chi(Rx)u\|_{H^{s_1}(\T^3)}\leq AR^{s_1}\|u\|_{H^{s_1}(\T^3)}.
$$	
\end{lem}
\begin{proof}
First for $s_1\in\N$, the proof follows from the direct calculation. For general $s_1\geq 0$, the conclusion follows from the interpolation. 
\end{proof}
Take $\chi\in C_c^{\infty}(\R^3)$, such that $\chi(x)\equiv 1$ if $|x|<\frac{1}{3}$ and $\chi\equiv 0$ if $|x|\geq \frac{1}{2}$. Define $\chi_k(x):=\chi((x-z^k)/r_k)$, hence $\chi_k|_{\widetilde{B}_k}\equiv 1$ and $\chi_k|_{(B(z^k,r_k/2))^c}\equiv 0$. Then \eqref{consequenceFPS} is translated to
$$ \chi_k(x)(u^{\epsilon_{n_k}}(t),\partial_tu^{\epsilon_{n_k}}(t) )=\chi_k(x)(\widetilde{u}_k(t),\partial_t\widetilde{u}_k(t) ),\quad \forall t\in[0,t_{n_k}].
$$
From Lemma \ref{restriction},
$$ \|u^{\epsilon_{n_k}}(t_{n_k})\|_{H^s(\T^3)}\gtrsim r_k^{s}\|\chi_ku^{\epsilon_{n_k}}(t_{n_k})\|_{H^s(\T^3)}\sim (\log\log n_k)^{-3s}\|\chi_k(x)\widetilde{u}_k(t_{n_k})\|_{H^s(\T^3)}.
$$
Therefore,
\begin{equation*}
\begin{split}
\|\chi_k(x)\widetilde{u}_k(t_{n_k})\|_{H^s(\T^3)}\geq & \|\widetilde{u}_k(t_{n_k})\|_{H^s(\T^3)}-\|(1-\chi_k)\widetilde{u}_k(t_{n_k}) \|_{H^s(\T^3)}\\
=&\|\widetilde{u}_k(t_{n_k})\|_{H^s(\T^3)}-\|(1-\chi_k)(\widetilde{u}_k(t_{n_k})-\widetilde{v}_k(t_{n_k}))\|_{H^s(\T^3)} ,
\end{split}
\end{equation*}
where in the last equality, we use the fact that $(1-\chi_k)\widetilde{v}_k(t_{n_k})=0$, thanks to the support property of $\widetilde{v}_k$. Therefore, we have
\begin{equation}\label{final-lower-bound}
\begin{split}
\|u^{\epsilon_{n_k}}(t_{n_k})\|_{H^s(\T^3)}\gtrsim &(\log\log n_k)^{-3s}\|\widetilde{u}_k(t_{n_k})\|_{H^s(\T^3)}-(\log\log n_k)^{-3s}
\|(1-\chi_k)\rho_{\epsilon_{n_k}}\ast S(t_{n_k})(u_0,u_1)\|_{H^s(\T^3)}\\
-&(\log\log n_k)^{-3s}\|(1-\chi_k)\big(\widetilde{u}_k(t_{n_k})-\rho_{\epsilon_{n_k}}\ast S(t_{n_k})(u_0,u_1)-\widetilde{v}_k(t_{n_k}) \big)\|_{H^s(\T^3)}.
\end{split}
\end{equation}
Applying Lemma \ref{restriction} again, we have
\begin{align}\label{final}
\|u^{\epsilon_{n_k}}(t_{n_k})\|_{H^s(\T^3)}\gtrsim (\log\log n_k)^{-3s}(\log n_k)^{s\sigma(\delta_2-\delta_1)-\delta_1}-\|\rho_{\epsilon_{n_k}}\ast S(t_{n_k})(u_0,u_1)\|_{H^s(\T^3)}-n_k^{-\theta}.
\end{align}
Since
$$ \|\rho_{\epsilon_{n_k}}\ast S(t_{n_k})(u_0,u_1)\|_{H^s(\T^3)}\lesssim 1,
$$
uniformly in $\epsilon_{n_k}$, 
by choosing $\delta_1>0$ small such that $s\sigma(\delta_2-\delta_1)-\delta_1>0$, the left hand side of \eqref{final} tends to $+\infty$ as $k\rightarrow\infty$. This completes the proof of Theorem \ref{thm:main}.
\end{proof}

Finally, we prove Corollary \ref{Gdelta}:
Recall the definition of the pathological set:
\begin{align*} \mathcal{P}:=\{(f,g)\in\mathcal{H}^s(\T^3):\,\,&\text{the solution } u^{\epsilon}(t) \text{ of \eqref{wave:main} with initial data } \rho_{\epsilon}\ast (f,g),\\ &
\text{satisfies the property }\quad
\limsup_{\epsilon\rightarrow 0} \|u^{\epsilon}(t)\|_{L^{\infty}([0,1];H^s(\T^3))}=+\infty \;  \}.
\end{align*}
For simplicity, below we will denote $u^{\epsilon}(t)=\Phi(t)(\rho_{\epsilon}\ast(f,g))$ the solution of \eqref{wave:main} with initial data $\rho_{\epsilon}\ast(f,g)$.   Obviously, the set
\begin{align*}
\mathcal{O}:=\{(f,g)\in\mathcal{H}^s(\T^3):&  \limsup_{k\rightarrow\infty}\|\Phi(t)(\rho_{\epsilon_{n_k}}\ast(f,g))\|_{L^{\infty}([0,1];H^s(\T^3) )}=\infty\}
\end{align*}
is contained in $\mathcal{P}$. From the proof of Theorem \ref{thm:main} in the last paragraph, $S\subset\mathcal{O}$, hence $\mathcal{O}$ is dense. It remains to show that $\mathcal{O}$ is a $G_{\delta}$ set, that is, a countable intersection of open sets. Note that $$\mathcal{O}=\bigcap_{N=1}^{\infty}\mathcal{O}_N, $$
where
\begin{align*}
	\mathcal{O}_N:=\{(f,g)\in\mathcal{H}^s(\T^3): \limsup_{k\rightarrow\infty}\|\Phi(t)(\rho_{\epsilon_{n_k}}\ast (f,g) )\|_{L^{\infty}([0,1];H^s(\T^3) )}>N \}.
\end{align*}
By definition, 
$$ \mathcal{O}_N=\bigcap_{k_0=1}^{\infty}\bigcup_{k=k_0}\mathcal{O}_{N,k},
$$
where
\begin{align*}
	\mathcal{O}_{N,k}:=\{(f,g)\in\mathcal{H}^s(\T^3): \|\Phi(t)(\rho_{\epsilon_{n_k}}\ast (f,g) )\|_{L^{\infty}([0,1];H^s(\T^3) )}>N \}.
\end{align*}
It suffices to show that, for fixed $N,k$, $\mathcal{O}_{N,k}$ is an open set. Indeed, pick $(f_0,g_0)\in\mathcal{O}_{N,k}$, denote by $$r_0:=\|\Phi(t)(\rho_{\epsilon_{n_k}}\ast (f_0,g_0) )\|_{L^{\infty}([0,1];H^s(\T^3))}-N>0.$$ From the inequality
$$ \|\rho_{\epsilon_{n_k}}\ast (f,g)\|_{\mathcal{H}^2(\T^3)}\leq C\epsilon_{n_k}^{-(2-s)}\|(f,g)\|_{\mathcal{H}^s(\T^3)},
$$
the Sobolev embedding $H^2(\T^3)\hookrightarrow L^{\infty}(\T^3)$,
and the global well-posedness theory in $\mathcal{H}^2(\T^3)$,we deduce that there exists a uniform constant $C_0>0$, such that
\begin{align}\label{comparison}
 &\sup_{t\in[0,1]}\|\Phi(t)(\rho_{\epsilon_{n_k}}\ast (f_0,g_0) )-\Phi(t)(\rho_{\epsilon_{n_k}}\ast (f,g) ) \|_{H^1(\T^3)}\\ \leq &C_0\epsilon_{n_k}^{-(2-s)(2\sigma+1)}(\|(f_0,g_0)\|_{\mathcal{H}^s(\T^3)}^{2\sigma}+\|(f,g)\|_{\mathcal{H}^s(\T^3)}^{2\sigma} )\|(f-f_0,g-g_0) \|_{\mathcal{H}^s(\T^3)}.
\end{align}
Choosing $$\delta<\frac{r_0\epsilon_{n_k}^{(2-s)(2\sigma+1)}}{2C_0}\cdot (1+\|(f_0,g_0)\|_{\mathcal{H}^s(\T^3)}^{2\sigma}+\|(f,g)\|_{\mathcal{H}^s(\T^3)}^{2\sigma})^{-1},
 $$
then if $\|(f,g)-(f_0,g_0)\|_{\mathcal{H}^s(\T^3)}<\delta$, by \eqref{comparison} we deduce that
$$ \|\Phi(t)(\rho_{\epsilon_{n_k}}\ast (f,g) )\|_{H^s(\T^3)}>N.
$$
This shows that $\mathcal{O}_{N,k}$ is open. The proof of Corollary \ref{Gdelta} is now complete.

%%%%%%%%%%%%%%%%%%%%%%%%%%%%%%%%%%%%%%%%%%%%%%%%%%%%%%%%%%%%%%%%%%%%%%%%%

\begin{flushright}
	
\end{flushright}
\end{document}